\documentclass[11pt,A4paper]{article}

\usepackage[left=25mm,right=25mm,top=25mm,bottom=30mm]{geometry}
\usepackage{epsfig}
\usepackage{epstopdf}
\usepackage{float}

\usepackage{xfrac}

\usepackage[percent]{overpic}

\usepackage{relsize}
\usepackage{color}
\usepackage{amsthm,amsmath,amssymb}
\usepackage{booktabs}
\usepackage{mathpazo}
\usepackage{microtype}
\usepackage[T1]{fontenc}
\usepackage{bm}
\usepackage{sectsty}
\usepackage[
	pdftitle={PDFTitle},
	pdfauthor={Hugo Parlier},
	ocgcolorlinks,
	linkcolor=linkred,
	citecolor=linkred,
	urlcolor=linkblue]
{hyperref}

\usepackage{mathtools}

\definecolor{linkred}{RGB}{199,21,133}

\definecolor{linkblue}{RGB}{75,0,130}

\usepackage[hang,flushmargin]{footmisc}
\usepackage{enumitem}
\usepackage{titlesec}
	\titlespacing{\section}{0pt}{12pt}{0pt}
	\titlespacing{\subsection}{0pt}{6pt}{0pt}

\titlelabel{\thetitle.\quad}

\makeatletter 

\long\def\@footnotetext#1{%
\H@@footnotetext{%
\ifHy@nesting 
\hyper@@anchor{\@currentHref}{#1}%
\else 
\Hy@raisedlink{\hyper@@anchor{\@currentHref}{\relax}}#1%
\fi 
}}

\def\@footnotemark{%
\leavevmode 
\ifhmode\edef\@x@sf{\the\spacefactor}\nobreak\fi 
\H@refstepcounter{Hfootnote}%
\hyper@makecurrent{Hfootnote}%
\hyper@linkstart{link}{\@currentHref}%
\@makefnmark 
\hyper@linkend 
\ifhmode\spacefactor\@x@sf\fi 
\relax 
}%

\ifFN@multiplefootnote%
\renewcommand*\@footnotemark{%
\leavevmode 
\ifhmode 
\edef\@x@sf{\the\spacefactor}%
\FN@mf@check 
\nobreak 
\fi 
\H@refstepcounter{Hfootnote}%
\hyper@makecurrent{Hfootnote}%
\hyper@linkstart{link}{\@currentHref}%
\@makefnmark 
\hyper@linkend 
\ifFN@pp@towrite 
\FN@pp@writetemp 
\FN@pp@towritefalse 
\fi 
\FN@mf@prepare 
\ifhmode\spacefactor\@x@sf\fi 
\relax%
}%
\fi 

\makeatother 

\theoremstyle{plain}
\newtheorem{theorem}{Theorem}[section]
\newtheorem{proposition}[theorem]{Proposition}
\newtheorem{lemma}[theorem]{Lemma}
\newtheorem{corollary}[theorem]{Corollary}

\makeatletter
\newtheorem*{rep@theorem}{\rep@title}
\newcommand{\newreptheorem}[2]{%
\newenvironment{rep#1}[1]{%
 \def\rep@title{#2 \ref{##1}}%
 \begin{rep@theorem}}%
 {\end{rep@theorem}}}
\makeatother

\newreptheorem{theorem}{Theorem}
\newreptheorem{corollary}{Corollary}

\theoremstyle{definition}
\newtheorem{definition}[theorem]{Definition}

\newcommand{\ii}{{\mathit i}}

\newcommand{\R}{{\mathbb R}}

\newcommand{\N}{{\mathbb N}}

\sectionfont{\large \bfseries}
\subsectionfont{\normalsize}

\long\def\symbolfootnote[#1]#2{\begingroup%
\def\thefootnote{\fnsymbol{footnote}}\footnote[#1]{#2}\endgroup}

\def\blfootnote{\xdef\@thefnmark{}\@footnotetext}

\begin{document}

{\Large \bfseries 
Ideally, all infinite type surfaces can be triangulated.}

{\large Alan McLeay and Hugo Parlier\symbolfootnote[1]{\small Both authors were supported by the Luxembourg National Research Fund OPEN grant O19/13865598.
}}

{\bf Abstract.} 
We show that any surface of infinite type admits an ideal triangulation. Furthermore, we show that a set of disjoint arcs can be completed into a triangulation if and only if, as a set, they intersect every simple closed curve a finite number of times. 
\vspace{1.0cm}

\section{Introduction}
Topological surfaces are often thought of as the result of pasting together polygons. Provided you have enough topology, pants decompositions are a natural way of decomposing (orientable) surfaces, or conversely one can build a surface by pasting together 3 holed spheres (pants) along their cuffs. This latter point of view has the advantage of being naturally related to underlying moduli spaces via the so-called length and twist Fenchel-Nielsen coordinates which parametrize the space of (marked) hyperbolic metrics of the topological surface (Teichm\"uller space). Alternatively, provided an underlying finite type surface has punctures, ideal triangulations are also used to parametrize the same space of metrics using shear coordinates. The corresponding topological construction is the process of pasting triangles along their sides and then forgetting the vertices which become ideal points (which geometrically correspond to cusps).

When passing to (orientable) surfaces of infinite type, the pants decomposition description is often used for describing how to construct these surfaces: one simply pastes together infinitely many pairs of pants to obtain a connected surface. This natural point of view has been very useful in the Teichm\"uller theory of these surfaces (see for instance \cite{AlessandriniEtAl,Basmajian-Saric, Saric}). When wanting to do the same thing for ideal triangulations, an immediate difficulty appears. In the finite type, an ideal triangulation is a maximal collection (with respect to inclusion) of disjoint simple arcs with end points in the ends \cite{Korkmaz-Papadopoulos, Penner}. In particular any collection of disjoint arcs can be completed into a triangulation. As we shall see, this no longer works in the infinite type case. 

In this short note, we observe that one can overcome this apparent difficulty.

\begin{theorem}\label{thm:main}
Any orientable surface of infinite type admits an ideal triangulation.
\end{theorem}

Before outlining the proof we make note of the following restriction for such ideal triangulations.

\begin{theorem}\label{thm:characterization}
An ideal multiarc is a subset of an ideal triangulation if and only if it intersects any simple closed curve a finite number of times.
\end{theorem}
Hence, an ideal triangulation is a multiarc maximal with respect to inclusion among all multiarcs which intersect any given simple closed curve a finite number of times. And as a consequence:
\begin{corollary}\label{cor:twoends}
Any ideal triangulation of an orientable surface must contain only finitely many arcs between any distinct pair of ends.
\end{corollary}
Following this observation, the first step in proving Theorem \ref{thm:main} is to show that we can triangulate a one-ended surface with countably many boundary components.  We then show that any infinite type surface can by decomposed into such subsurfaces.


\section{Setup}
By an orientable surface $X$ of infinite type, we mean an orientable topological surface with infinitely generated fundamental group. The homeomorphism type of a such a surface is determined by its genus and space of ends.

An end is a nested sequence $\{U_i\}_{i\in \N}$ of subsurfaces of $X$ (so $U_i\supset U_{i+1}$). Two ends $\{U_i\}_{i\in \N}$, $\{V_j\}_{j\in \N}$ are equivalent if for any $i_0$, there exists a $j_0$ such that $U_{i_0} \subset V_{j_0}$ and vice versa. 

If each $U_i$ contains genus then we say the end is \emph{nonplanar}, and \emph{planar} otherwise.  An end $e = \{U_i\}_{i\in \N}$ is \emph{isolated} if there exists some subsurface $U_{i_0}$ such that every nested sequence containing $U_{i_0}$ is equivalent to $e$.  Usually, we call an isolated planar end a puncture or cusp.

We adopt the standard definition of an arc between ends (see for instance \cite{DGGV}, and note this is more general than the definition in \cite{Fossas-Parlier} which is adapted for flip graphs).

\begin{definition}
An arc of $X$ is a proper embedding $\alpha:\R \to X$ such that the each of its ends correspond to an end of $X$.  Note that since $X$ is orientable, each arc has exactly two sides.
\end{definition}

Arcs play an important role in dynamical and topological considerations on infinite type surfaces, see for instance 
\cite{Bavard,FGM}. The definition above means that there exist ends $\{U_i\}_{i\in \N}$ and $\{V_j\}_{j\in \N}$ (not necessarily distinct) such that for all $i_0$, there exists $t_0 \in \R$ such that $\alpha(t) \in U_{i_0}$ for all $t>t_0$, and there exists $\{V_j\}_{j\in \N}$ such that for all $j_0$, there exists $s_0$ with $\alpha(t) \in V_{j_0}$ for all $t<s_0$. An arc is non-trivial if its complementary region does not contain a topological disk. This is equivalent to there not existing an end $\{U_i\}_{i\in \N}$ such that for any $i_0 \in \N$, the arc can be freely homotoped to an arc lying entirely inside $U_{i_0}$.

Curves (or simple closed curves) are much easier to define: they are just embeddings of $S^1$. They are non-trivial if they are not freely homotopic to a point or a puncture.

We are only interested in curves and arcs up to (free) homotopy. 

A pants decomposition is a collection of dijoint curves whose complementary regions are homeomorphic to thrice punctured spheres. A corollary of Richards' classification theorem for infinite type surfaces \cite{Richards} is the following:
\begin{corollary}
Any surface $X$ can be obtained by pasting together pants. Equivalently, any surface can be decomposed into a pants decomposition by cutting along disjoint simple closed curves.
\end{corollary}

Observe that dual to a pants decomposition is a (possibly infinite) trivalent (or cubic) graph. 

\begin{definition}
An ideal triangulation of $X$ is a collection of arcs $\mu = \{\mu_i\}_{i \in \N}$ such that each component of $X \setminus \mu$ is an open disk and each disk is bounded by at most three sides of arcs. 
\end{definition}

We end this section with a triangulation of the flute surface (see \cite{Basmajian}).

\begin{proposition}\label{prop_flute}
There exists an ideal triangulation of the flute surface.
\end{proposition}

\begin{proof}
First we take the standard Farey triangulation $T$ of an open disk $D$. By removing a single point from each triangle of $T$ we arrive at a surface $F \subset D$ homeomorphic to the flute surface (see Figure \ref{fig:puncturedFarey}). 

\begin{figure}[h]
\begin{center}
\includegraphics[width=8cm]{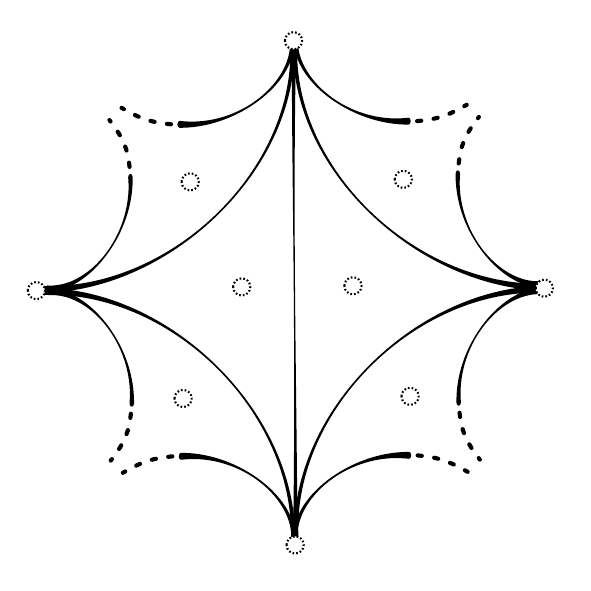}
\vspace{-24pt}
\end{center}
\caption{The punctured Farey triangulation}
\label{fig:puncturedFarey}
\end{figure}

While the projection of $T$ onto $F$ is not a triangulation, it can be extended to one by triangulating each connected component of $F \setminus T$.  Each such component is homeomorphic and can be easily triangulated as in Figure \ref{fig:punctured-triangle}.
\end{proof}

\begin{figure}[H]
\begin{center}
\includegraphics[width=5cm]{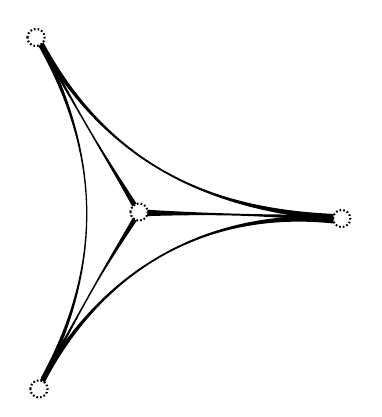}
\vspace{-24pt}
\end{center}
\caption{A triangulation of a punctured triangle}
\label{fig:punctured-triangle}
\end{figure}
\section{Triangulations galore}

Our construction for more general triangulations draws inspiration from the example above.

\subsection{Perforated Farey Triangulation}

To obtain a perforated Farey triangulation, you start with the Farey triangulation of the hyperbolic plane and you add a puncture or a boundary curve to a collection of triangles (see Figures \ref{fig:PerforatedFarey} and \ref{fig:triangles}), and so that each arc leaves a puncture or a boundary curve on either side.

\begin{figure}[h]
\begin{center}
\includegraphics[width=6cm]{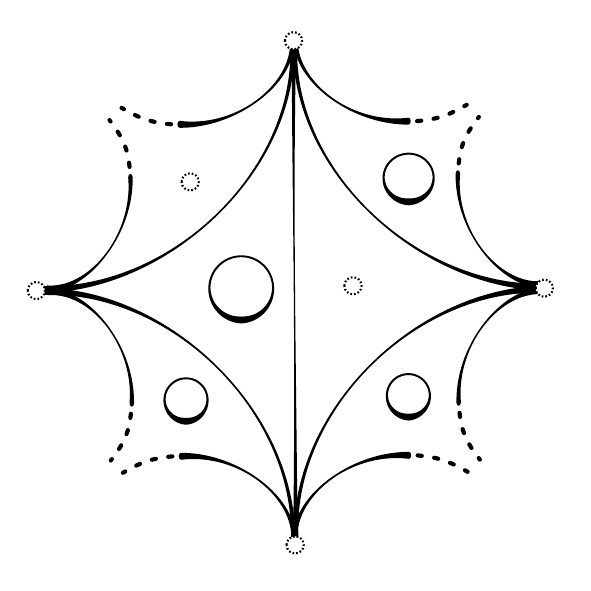}
\vspace{-24pt}
\end{center}
\caption{A perforated Farey triangulation}
\label{fig:PerforatedFarey}
\end{figure}

\begin{figure}[h]
\begin{center}
\includegraphics[width=9cm]{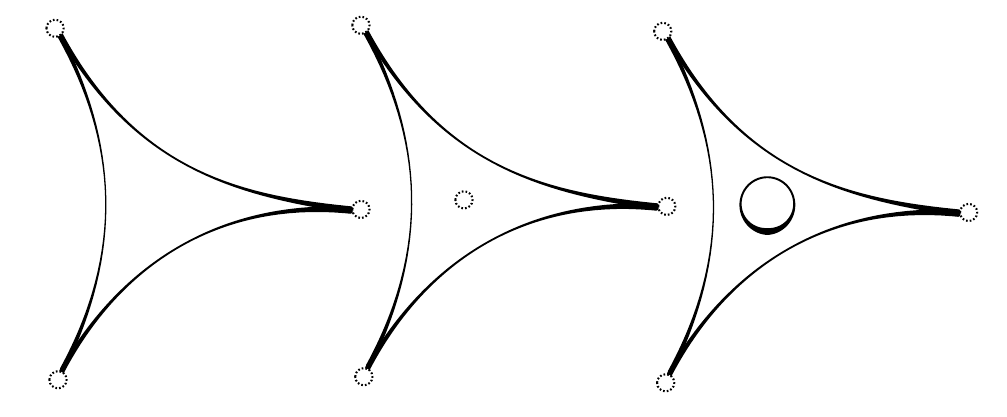}
\vspace{-24pt}
\end{center}
\caption{Triangles may be left as triangles, have a puncture, or a boundary curve}
\label{fig:triangles}
\end{figure}

Note that with this condition, no two arcs are isotopic and any arc leaves infinitely many punctures or boundary curves on either of its sides.

From now on we will use the term flute surface to refer to any surface obtained in this way.  That is, our flute surfaces may have compact boundary components as opposed to just punctures.



The following lemma is illustrated by Figure \ref{fig:six}.

\begin{lemma}\label{lem_twotriangles}
Any two one holed triangles pasted together along their holes can be triangulated using 6 arcs.
\end{lemma}

\begin{figure}[h]
\begin{center}
\includegraphics[width=3cm]{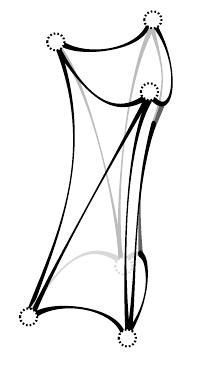}
\vspace{-24pt}
\end{center}
\caption{Triangulating a cylinder with three ideal vertices on each boundary}
\label{fig:six}
\end{figure}

A key observation for our purposes is that any surface can be decomposed into subsurfaces homeomorphic to flute surfaces.

\begin{lemma}\label{lem_flutedecomp}
Any pants decomposition contains a collection of curves $\Gamma$ such that $X\setminus \Gamma$ is a disjoint union of flute surfaces.
\end{lemma}
\begin{proof}
Consider $P$, a pants decomposition of $X$. Let $G_P$ be its dual graph. The edges of $G_P$ are the curves of $P$, and the vertices are the pants. The degree of every vertex is $3$, except for vertices which correspond to pants with one or two boundary curves which are isolated planar ends, and whose degree is either two or one. 

Now choose a spanning tree $T_0$ of $G_P$. Every vertex of $G_P$ belongs to $T_0$, but certain edges do not. Let $E_0$ be the set of edges of $G_P$ that do not belong to $T$ and let $\Gamma_0$ be the corresponding curves of $P$. We set $X_0=X\setminus \Gamma_0$. Note that by construction $X_0$ is of genus $0$, and we have a subset of $P$ that is a pants decomposition of $X_0$ and whose dual graph is $T_0$. Vertices of $T_0$ are either of degree $1$, $2$ or $3$: those of degree $1$ and $2$ correspond to pants bounding either isolated planar ends, or curves in $\Gamma_0$. 

The next step in the construction will be iterated.

If $T_0$ is one ended, then $X_0$ is a flute surface. If not, it contains an infinite geodesic $c_1$ (for the graph metric). As $T_0$ is a tree, each of its edges is separating. If an edge does not separate two infinite trees, we call it irrelevant. Otherwise it is relevant. Note that for instance if $T_0$ is an infinite trivalent tree, every edge is relevant. On the contrary, if $T_0$ is one ended, all edges are irrelevant.

We consider the set of all edges of $T_0$ that are both adjacent to $c_1$ and relevant and denote by $\Gamma_1\subset \Gamma_0$ the corresponding multicurve. Note that $X_0\setminus \gamma_1 = U\sqcup V$ where $U$ is homeomorphic to a two ended flute surface. We set $X_1=V$ and repeat the above procedure, possibly infinitely many times. 

The end result of the procedure is a multicurve $\Gamma'=\cup_{i \in I} \Gamma_i\subset P$ where $I$ is an index set (possibly empty, possibly infinite) such that $X\setminus \Gamma'$ is a collection of one or two ended flute surfaces. To obtain a collection of one ended flute surfaces, it suffices to add to $\Gamma'$ a single relevant curve for each of the two end ended flute surfaces, which results in the multicurve $\Gamma$ as desired.
\end{proof}

\subsection{Triangulating any surface}

Here we prove the following, which is a slight generalization of Theorem \ref{thm:main}.

\begin{theorem}\label{thm_restated}
Let $X$ be an infinite type surface with boundary consisting in a collection of simple closed curves. Then there exists an ideal multiarc $T$ such that $X\setminus T$ is a collection of triangles or one holed monogons. 
\end{theorem}

\begin{proof}
Let $\Gamma$ be the collection of curves described by Lemma \ref{lem_flutedecomp} and let $F$ be one of the flute components of $X \setminus \Gamma$.  We then take a perforated Farey triangulation of $F$ and extend it to $F$ by adding three arcs for each isolated planar end, and a single arc for each boundary component of $X$.  Denote by $T'$ the resulting multiarc.

By construction each non-triangular component of $X \setminus T'$ is either a one holed monogon or two triangles pasted together along a simple closed curve.  By Lemma \ref{lem_twotriangles} we can extend $T'$ to a multiarc $T$ satisfying the properties of the theorem.
\end{proof}

\section{Completing ideal multiarcs into triangulations}

We can now prove Theorem \ref{thm:characterization} which we restate for convenience.
\begin{theorem} A collection of arcs $\mu_0$ can be completed into a triangulation if and only if $\ii(\mu_0,\alpha) < \infty$ for any simple closed curve $\alpha$.
\end{theorem}

\begin{proof}
First, assume $\alpha$ and $\mu_0$ have infinite intersection.  Let $z$ be a point in $\alpha$ accumulated by intersection points with $\mu_0$.  Now, if $\mu$ is a triangulation of $X$ containing $\mu_0$ then $z$ cannot belong to any disk in $X \setminus \mu$, as this would imply infinitely many arcs in $\mu_0$ intersect this disk.  Similarly, $z$ cannot belong an arc $\delta$ in $\mu$, as this would imply that infinitely many arcs of $\mu_0$ intersect the union of the two disks bounded by $\delta$.

To prove the other direction, we construct a triangulation containing $\mu_0$. We begin by looking at $X\setminus \mu_0$ to see if it contains any finite type pieces (hence bounded by arcs in $\mu_0$). If so, we triangulate the pieces that are not already triangles, and we set $X_0$ to be this collection of triangles. If not, $X_0$ is the empty set.

Let $P$ be a pants decomposition of $X$. The main idea is to add arcs to $\mu_0$ such that they contain triangles that contain every curve in $P$. To do so we give curves in $P$ an order (so $P= \cup_{i\in \N}\alpha_i$). Now we proceed inductively to construct a multiarc $\mu_i$, and a subsurface $X_i$, which will be triangulated by arcs of $\mu_i$ and which contains all curves $\alpha_1, \hdots, \alpha_i$. By construction $\mu_i$ will contain $\mu_{i-1}$. The induction starts with $X_0$ and $\mu_0$.

We now describe the inductive step. Let $i>0$. There are two cases for $\alpha_i$, depending on whether it intersects $\mu_{i-1}$. 

\underline{Case 1:} $\alpha_i$ and $\mu_{i-1}$ are disjoint.

First, suppose $Y$ is a component of $X \setminus \alpha_i$ such that $Y$ and $\mu_{i-1}$ do not intersect. (For this to happen, note that $\alpha_i$ must be a separating curve.) Then from Theorem \ref{thm_restated} we can construct a multiarc in $Y$ whose complement is a collection of triangles and single one holed monogon containing $\partial Y$. We can add this multiarc to $\mu_{i-1}$ and look at another component.

Note that unless $\mu_{i-1} = \emptyset$, at least one component of $X\setminus \alpha_i$ must intersect $\mu_{i-1}$. 

We now consider $Y$, a component of $X \setminus \alpha_i$ which intersects $\mu_{i-1}$. Suppose it has a connected boundary (hence $\alpha_i$ is again separating). Let $a$ be an embedding of $[0,1]$ such that $a \cap \mu_{i-1} = a(0)$, and $a(1)\in \partial Y$. Now, let $\gamma$ be the arc of $\mu_{i-1}$ intersecting $a$ and let $b$ be a component of $\gamma \setminus a$.  Finally, define $\gamma_1$ be the boundary of a regular neighborhood of $a \cup b \cup \partial Y$.  Note that, up to homotopy, $\gamma_1$ is an arc disjoint from $\mu_{i-1}$ and $Y \setminus \gamma_1$ contains a one holed monogon.

If $Y$ contains arcs in $\mu_{i-1}$ and does not have connected boundary (and in particular $\alpha_i$ is nonseparating) then add $\gamma_1$ to $\mu_{i-1}$ and repeat the above step in order to define at an arc $\gamma_2$ associated to the second boundary component.

We therefore arrive at two arcs $\gamma_1$, $\gamma_2$ that are disjoint from $\mu_{i-1}$ such that the component $Z$ of $X \setminus \{\gamma_1, \gamma_2\}$ containing $\alpha_i$ is of finite type.  Let $\zeta_{i-1}$ be a triangulation of $Z$.  We then define $\mu_i$ to be the union of $\mu_{i-1}$, $\gamma_1$, $\gamma_2$, and $\zeta_{i-1}$.  We appropriately define $X_i$ to be the subsurface spanned by $\mu_i$, that is, $X_i = X_{i-1} \cup Z \cup \{\gamma_1, \gamma_2\}$.

\underline{Case 2:} $\alpha_i$ and $\mu_{i-1}$ intersect.

Let $\Gamma = \alpha_i \cup \mu_{i-1}$ and let $N(\Gamma)$ be a regular neighborhood of $\Gamma$. Define $A(\Gamma)$ to be the set of essential arcs in $\partial N(\Gamma)$ and define $C(\Gamma)$ to be the set of essential curves in $\partial N(\Gamma)$.  Note that since $i(\mu_{i-1}, \alpha_i) < \infty$ both $C(\Gamma)$ and $A(\Gamma)$ are finite.  For each curve in $C(\Gamma)$, repeat the argument from case 1 to arrive at a multiarc $\mu_{i-1}'$ and a subsurface $X_{i-1}'$.  Now, there exists a subsurface $X_i \subset X$ such that $X_{i-1}' \cup A(\Gamma) \subset X_i$ and $X_i \setminus X_{i-1}'$ is finite type.  We can therefore extend $\mu_{i-1}'$ to a triangulation $\mu_i$ of $X_i$.  Note that $X_i$ containts $\alpha_i$ as desired.

The surfaces $\cup_{i\in \N} X_i$ clearly form an exhaustion of $X$. We set $T=\cup_{i\in \N} \mu_i$, which by construction is thus a triangulation of $X$. 
\end{proof}

{\em Address:}\\
Department of Mathematics, University of Luxembourg, Luxembourg \\
{\em Email:}
 \href{mailto:alan.mcleay@unilu.ch}{alan.mcleay@uni.lu}, \href{mailto:hugo.parlier@unilu.ch}{hugo.parlier@uni.lu}\\

\end{document}